\documentclass[12pt]{article}
\usepackage{amssymb}
\usepackage{amsfonts}
\usepackage{amsmath}
\usepackage[usenames]{color}
\usepackage{mathrsfs}
\usepackage{amsfonts}
\usepackage{amssymb,amsmath}
\usepackage{CJK}
\usepackage{cite}
\usepackage{cases}
\usepackage{amsthm}
\usepackage{amssymb}
\usepackage{graphicx}
\usepackage{amsmath}
\usepackage[all]{xy}
\usepackage{enumerate}
\usepackage{amscd}
\usepackage{cases}

\def\cl{\centerline}

\def\G{\mathcal{G}}
\def\b{\beta}
\def\vs{\vspace*}
\def\S{\mathcal{S}}
\def\U{\mathcal{U}}

\def\V{\mathcal{V}}
\def\B{\mathcal{B}}
\def\M{\mathfrak{M}}
\def\Z{\mathbb{Z}}
\def\N{\mathbb{N}}
\def\b{\mathfrak{b}}

\def\S{\mathcal{S}}
\def\P{\mathcal{P}}
\def\W{\mathfrak{W}}
\def\R{\mathcal{R}}
\def\G{\mathcal{G}}
\def\B{\mathcal{B}}
\def\C{\mathbb{C}}

\def\X{\mathbb{X}}

\def\ni{\noindent}

\numberwithin{equation}{section}
\newtheorem{theo}{Theorem}[section]
\newtheorem{defi}[theo]{Definition}

\newtheorem{lemm}[theo]{Lemma}
\newtheorem{prop}[theo]{Proposition}
\newtheorem{clai}{Claim}

\begin{document}
\begin{center}
\cl{\large\bf \vs{8pt}  Simple restricted modules over the       $N=1$ Ramond algebra}
\cl{\large\bf \vs{8pt} as weak modules   for vertex operator superalgebras}
\cl{ Haibo Chen}
\cl{Dedicated to my supervisor Professor Yucai Su on his 60th birthday}
\end{center}
\footnote {
H. Chen (hypo1025@163.com).
}
{\small
\parskip .005 truein
\baselineskip 3pt \lineskip 3pt

\noindent{{\bf Abstract:}
In the present paper,   a class of new simple modules over the    $N=1$ Ramond algebra are constructed, which  are induced from simple modules over some finite dimensional solvable Lie superalgebras.
These new    modules  are simple       restricted   modules over the $N=1$    Ramond algebra.   Combined with the result in \cite{L1}, a classification of simple weak $\psi$-twisted $\bar W(0,c)$-modules under certain conditions is also given. At last, some examples of simple restricted   $N=1$    Ramond  modules as  various versions of Whittaker modules are presented (classical Whittaker modules were studied in \cite{LPX}).
\vs{5pt}

\ni{\bf Key words:}
  $N=1$ Ramond algebra,  restricted module,  simple module,  vertex operator superalgebra.}

\ni{\it Mathematics Subject Classification (2020):} 17B10, 17B65, 17B68, 17B69.}
\parskip .001 truein\baselineskip 6pt \lineskip 6pt
\section{Introduction}
 \textit{The  $N=1$ Ramond algebra}   is an  infinite dimensional  Lie superalgebra
$$\R=\bigoplus_{m\in\Z}\C L_m \oplus\bigoplus_{m\in\Z}\C G_{m} \oplus\C C,$$
which  satisfies  the following Lie super-brackets
 \begin{equation}\label{def55441.1}
\aligned
&[L_m,L_n]= (n-m)L_{m+n}+\frac{n^{3}-n}{12}\delta_{m+n,0}C,\\&
 [G_m,G_n]= -2L_{m+n}+\frac{4m^{2}-1}{12}\delta_{m+n,0}C,\\&
 [L_m,G_n]= (n-\frac{m}{2})G_{m+n},\ [\R,C]=0,
\endaligned
\end{equation}
where $m,n\in\Z$.
By definition, we have the following decomposition:
$$\R=\R_{\bar 0}\oplus\R_{\bar 1},$$
where $\R_{\bar 0}=\mathrm{span}\{L_m,C\mid m\in\Z\}$, $\R_{\bar 1}=\mathrm{span}\{G_m\mid m\in\Z\}$.
  Notice that the even part  $\R_{\bar 0}$   is isomorphic to
 the  classical Virasoro algebra.
   Clearly, $\C C$ is the center of $\R$.
Let $\R_m=\mathrm{span}\{L_{m},G_{m},\delta_{m,0}C\}$ for all $m\in\Z$.
Then $[\R_m,\R_n]\subset\R_{m+n}$, and $\R$,  $\U(\R)$ are $\Z$-graded.
 It is easy to see that $\R$ has the following triangular decomposition:
$$\R=\R_+\oplus\R_0\oplus\R_-,$$ where $\R_+=\mathrm{span}\{L_{m},G_{m}\mid m\in\Z_+\}$ and $\R_-=\mathrm{span}\{L_{m},G_{m}\mid -m\in\Z_+\}$.

The   representation theory  of the $N=1$ Ramond algebra   has   attracted
a lot of attention from many researchers.  The structure of Verma modules       over the $N=1$ Ramond algebra  was  investigated  by some physicists and mathematicians (see, e.g., \cite{D1,D2,D3,IK}).  The structures of   Fock modules
and pre-Verma  modules    for  the $N=1$ Ramond algebra  were respectively studied   in \cite{IK1,IK2}.  All  simple Harish-Chandra modules over the  $N=1$ Ramond algebra  were
classified  in \cite{S}  (also see \cite{CLL}).
 Recently, some non-weight modules as Whittaker modules and  $\U(\C L_0\oplus \C G_0)$-free modules of rank $1$ over the    $N=1$ Ramond algebra were studied in \cite{LPX,YYX2}, respectively. Moreover, a class of  non-weight  modules over  the  $N=1$ Ramond algebra
were developed in \cite{CDL}, which include  super intermediate series modules, $\U(\C L_0)$-free modules of rank $2$ and so on.
Clearly,  the $N=1$ Ramond algebra can be seen as a certain supersymmetric extension  of
the Virasoro algebra.   However, the representation theory of the $N=1$ Ramond algebra is far less abundant than the Virasoro algebra.

Highest weight modules and Whittaker modules are two classes of classical representation of Lie (super)algebra, which  both  belong  to the category of restricted modules.
In \cite{MZ},    a generalized construction for simple Virasoro modules was given by Mazorchuk and Zhao, which  included highest weight modules  and various version
 Whittaker modules. These simple Virasoro modules  were     restricted modules.  In \cite{L0,CCHS},
 a class of weak modules over the Virasoro vertex operator algebra
$V(c,0)$  were classified  by  the restricted Virasoro modules of the same level. From then on,
to  enrich the representation theory of infinite dimensional Lie (super)algebras, the restricted modules
over some other Lie (super)algebras such as twisted (mirror) Heisenberg-Virasoro algebras,  $N=1$ Neveu-Schwarz algebras, gap-$p$ Virasoro algebras, affine Lie algebras $A_1^{(1)}$ were   investigated (see, e.g., \cite{TYZ,G,LPX55,LPXZ,CGHS,L,CG,ALZ}). In  Proposition 4.1 of \cite{L1}, they showed that   any restricted module for
the $N=1$ Ramond algebra   of central charge $c$ is a weak $\psi$-twisted $\bar{W}(0,c)$-module.
In order to determine those weak modules,  we  need to give a characterization  for   restricted modules over
the $N=1$ Ramond algebra.  Motivated by this, we consider   writing this paper.

The rest of this paper is organized as follows.
In Section $2$,   some definitions and notations of restricted modules are presented for later use.
In Section $3$, a class of new simple  restricted modules over the    $N=1$  Ramond algebra  are constructed in Theorem \ref{th1}.
In Section $4$, we give a characterization of simple restricted $\R$-modules under certain conditions in Theorem \ref{th2}, which reduces the problem of classification of simple restricted
$\R$-modules to classification of simple modules over a class of finite-dimensional solvable Lie superalgebras. In Section $5$,    we see that those induced simple $\R$-modules defined in Theorem \ref{th1} are weak $\psi$-twisted $\bar W(0,c)$-modules.  Finally,   some examples of restricted $\R$-modules are given, such as   Whittaker
modules  and high order Whittaker modules.

Throughout this paper,  we denote by $\C$,  $\Z$, $\N$ and $\Z_+$ the sets of complex numbers, integers,
 nonnegative integers and positive integers, respectively.
    All vector superspaces
(resp. superalgebras, supermodules)   and
 spaces (resp. algebras, modules)    are considered to be over
$\C$. We use $\U(\mathfrak{a})$ to denote
the universal enveloping algebra for a Lie (super)algebra $\mathfrak{a}$.
\section{Preliminaries}
Let  $M=M_{\bar0}\oplus M_{\bar1}$ be a $\Z_2$-graded vector space. We say that any element $v\in M_{\bar0}$ (resp. $v\in M_{\bar1}$)
is    {\em even}    (resp. {\em odd}). Set $|v|=0$ if
$v\in M_{\bar0}$ and $|v|=1$ if $v\in M_{\bar1}$. We call that all elements  in  $M_{\bar0}$ or  $M_{\bar1}$ are   {\em homogeneous}.
 Throughout this paper, all elements in Lie superalgebras and modules are homogenous.      All modules for Lie superalgebras   are $\Z_2$-graded.
All simple modules over Lie (super)algebras are  non-trivial  unless
specified.

Let $\G$ be a Lie superalgebra. {\em A  $\G$-module} is a $\Z_2$-graded vector space $M$ together
with a bilinear map $\G\times M\rightarrow M$, denoted $(x,v)\mapsto xv$ satisfying the following conditions
$$x(yv)-(-1)^{|x||y|}y(xv)=[x,y]v,\
\G_{\bar i} M_{\bar j}\subseteq  M_{\bar i+\bar j}$$
for all $x, y \in \G, v \in M$. It is clear that there is a parity-change functor $\Pi$ on the category of
$\G$-modules to itself.
\begin{defi}\rm
Assume   that $M$  is a  $\G$-module   and $x\in\G$.

\begin{itemize}
\item[{\rm (i)}] If for any $v\in M$ there exists $m\in\Z_+$ such that $x^mv=0$,  we say that   the action of  $x$  on $M$ is {\em locally nilpotent}.
Similarly, if for any $v\in M$ there exists $m\in\Z_+$ such that $\G^mv=0$, we say that  the action of $\G$   on $M$  is  {\em locally nilpotent}.

\item[{\rm (ii)}]  If for any $v\in M$  we have  $\mathrm{dim}(\sum_{m\in\Z_+}\C x^mv)<+\infty$,  we call that the   $x\in\G$ acts  {\em locally finite} on $M$.
Similarly, if
for any $v\in M$ we get $\mathrm{dim}(\sum_{m\in\Z_+} {\G}^mv)<+\infty$,
 we call that  the  $\G$  acts  {\em locally finite}  on $M$.
 \end{itemize}
\end{defi}
The action of $x$ on $M$ is locally  nilpotent, this   implies that  the action of $x$ on $M$ is locally finite.  Generally speaking, it is not true for any Lie (super)algebra $\G$. But
if $\G$ is a finitely generated Lie (super)algebra,   that  the action of
 $\G$
 on $M$ is locally  nilpotent  implies that   the action of $\G$ on $M$ is locally finite.

\begin{defi}\rm
Assume that  $\G=\bigoplus_{m\in\Z}\G_m$ is  a $\Z$-graded Lie superalgebra. A $\G$-module $M$ is called the
{\em restricted} module if for any $v\in M$ there exists $k\in\N$ such that $\G_mv=0$  for $m>k$.
\end{defi}

For simplicity,    write $\X=\{0,1\}$.  We denote by   $\S$   the set of all infinite vectors of the form $\mathbf{i}:=(\ldots, i_2, i_1)$ with   $i_{2m-1}\in\N,i_{2m}\in \X,m\in\Z_+$,
satisfying the condition that the number of nonzero entries is finite.
For
$k\in\Z_+$, write  $\epsilon_k=(\ldots,\delta_{k,3},\delta_{k,2},\delta_{k,1})$   and $\mathbf{0}=(\ldots, 0, 0)$. Denote
$$\W(\chi)=m \quad  \mathrm{if}\quad  0\neq \chi\in \U(\R)_{-m},\quad \forall m\in\N.$$  For    $\mathbf{i}\in \S$,   set
\begin{eqnarray}\label{w12112}
g_{\mathbf{i}}=\cdots \big(G_{-k+1}^{i_{2k}}L_{-k}^{i_{2k-1}}\big) \cdots \big(
G_{-2}^{i_6}L_{{-3}}^{i_{5}}\big)\big(
G_{-1}^{i_4}L_{{-2}}^{i_{3}}\big)\big(G_0^{i_2}L_{-1}^{i_1}\big)\in \U(\R_-\oplus \R_0),
\end{eqnarray}
where $i_{2k-1}\in\N,i_{2k}\in\X, k\in\Z_+$.
Then we have $$\W(\mathbf{i}):=\W(g_{\mathbf{i}})
=\sum_{k=1}^{+\infty}ki_{2k-1}+\sum_{k=1}^{+\infty}(k-1)i_{2k}.$$
Denote
$$\mathbf{D}(\mathbf{i}):=\mathbf{D}(g_{\mathbf{i}})
=\sum_{k=1}^{+\infty}(i_{2k-1}+i_{2k}).
$$

The following total order on  $\S$ can be found in \cite{MZ}.
\begin{defi}\rm\label{defin421}
 We denote by $>$ the   {\em reverse  lexicographical  total order}  on  $\S$,  defined as follows:
\begin{itemize}
\item[{\rm (a)}]
  $\mathbf{0}$ is the minimum element;
 \item[{\rm (b)}]
  for different nonzero $\mathbf{i},\mathbf{j}\in\S$, we have
$$\mathbf{j} > \mathbf{i} \ \Longleftrightarrow
 \ \mathrm{ there\ exists} \ m\in\Z_+ \ \mathrm{such \ that} \ (j_k=i_k,\ \forall 0< k<m) \ \mathrm{and} \ j_m>i_m.$$
 \end{itemize}
\end{defi}
\begin{defi}\rm
By using the above reverse lexicographical total order, we can   define the   {\em  principal total order}  $\succ$ on $\S$: for different  $\mathbf{i},\mathbf{j}\in \S$, set $\mathbf{i}\succ\mathbf{j}$ if and only if one of the following condition is satisfied:
\begin{itemize}
\item[{\rm (a)}]  $\W(\mathbf{i})> \W(\mathbf{j})$;

\item[{\rm (b)}] $\W(\mathbf{i})= \W(\mathbf{j})$ and $\mathbf{D}(\mathbf{i})> \mathbf{D}(\mathbf{j})$;

    \item[{\rm (c)}] $\W(\mathbf{i})= \W(\mathbf{j})$, $\mathbf{D}(\mathbf{i})= \mathbf{D}(\mathbf{j})$
        and $\mathbf{i}>\mathbf{j}$.
\end{itemize}
\end{defi}
For any   simple module  $V$ over  $\R$ or one of its subalgebra   containing the central element $C$,
  we denote that  the action  of $C$ on   $V$ is a scalar  $c$.
  Write $\B:=\R_+\oplus\C L_0\oplus \C C$.  Let $V$ be a simple $\B$-module.
 We have the following  induced   $\R$-module
$$\mathrm{Ind}_{\B,c}(V)=\U(\R)\otimes_{\U(\B)}V.$$
By the $\mathrm{PBW}$ Theorem (see \cite{CW}), and $G_{-m}^2=-L_{-2m}$  for $m\in\N$, every element of $\mathrm{Ind}_{\B,c}(V)$ can be uniquely written as
the following form
\begin{equation}\label{def2.1}
\sum_{\mathbf{i}\in\S}g_{\mathbf{i}} v_{\mathbf{i},c}
\end{equation}
where $g_{\mathbf{i}}$ defined as \eqref{w12112},    $v_{\mathbf{i},c}\in V$ and only finitely many of them are nonzero. For any $w\in\mathrm{Ind}_{\B,c}(V)$ as in  \eqref{def2.1}, we write  $\mathrm{supp}(w)$ the set of all $\mathbf{i}\in \S$  such that $v_{\mathbf{i},c}\neq0$.
 For   $0\neq w\in \mathrm{Ind}_{\B,c}(V)$,
 we denote by $\mathrm{deg}(w)$  the maximal element in $\mathrm{supp}(w)$ (with
respect to  the principal total order on $\S$),
  which is called the {\em degree} of $w$. Note that $\mathrm{deg}(w)$ is defined only for  $w\neq0$.
Let
$\W(w)=\W(\mathrm{deg}(w))$ and $\W(0)=-\infty$.
For any $m\in\N$, set
$$\mathrm{supp}_m(w)=\{\mathbf{i}\in \mathrm{supp}(w)\mid \W(\mathbf{i})=m\}.$$

\section{Construction of   simple  $\R$-modules}
For $k\in\Z$, denote $\Z_{\geq k}=\{m\in\Z\mid m\geq k\}$. Now  we give a characterization for simple $\B$-modules.
\begin{lemm}\label{lemm31}
Let $k,m\in\Z, t\in\Z_{\geq2}$ and     $V$ be a simple   $\B$-module with   $L_{k}V=0$ for all $k> t$.
Then we get
  $G_{m}V=0$ for all $m> t$.
\end{lemm}
\begin{proof}
Choose    $m\in\Z$ with $m\geq t+1$.
 According to $G_{m}^2V=-L_{2m}V=0$, we know that  $\Upsilon=G_{m}V$ is a subspace
of $V$ and $\Upsilon\neq V$. For any $l\in\Z_+$, one gets
$$G_{m+l}V=\frac{2}{3-m-l}(L_{m+l-1}G_{1}-G_{1}L_{m+l-1})V=0.$$
Then for any $n\in\N$,   we check
$$ L_{n}\Upsilon=L_{n}G_{m}V=G_{m}L_{n}V+(m-\frac{n}{2})G_{m+n}V\subset \Upsilon.$$
For $k\in\Z_+$, it is clear that  $G_{k}\Upsilon \subset \Upsilon.$
Thus,  $\Upsilon$ is a  submodule of $V$. It follows from  the simplicity of $V$ that we get $\Upsilon=G_{m}V=0$ for all $m> t$.
\end{proof}

\begin{lemm}\label{lemm3231}
Let $m, k\in\Z, c\in\C, F_k=
L_k\ \mathrm{or}\ G_k$.   Let  $V$ be a simple   $\B$-module and there exists $r\in\Z_{\geq2}$  such that $V$   satisfying  the following two conditions  \begin{itemize}
\item[{\rm (1)}] the action of $L_{r}$ on   $V$  is  injective;
\item[{\rm (2)}]  $L_{m}V=G_rV=0$ for all $m>r$.
\end{itemize}
  Then for any $0\neq w\in\mathrm{Ind}_{\B,c}(V)$ with $\W(w)=q\in \Z$, $k\geq r$, we get
 \begin{itemize}
\item[{\rm (i)}]  $\mathrm{supp}_{q-k+r}(F_k w)\subset\big\{\mathbf{i}-\mathbf{j}\mid \mathbf{i}\in\mathrm{supp}_{q}(w),\W(\mathbf{j})=k-r\big\}$;
    \item[{\rm (ii)}] $\W(F_k w)\leq q-k+r$.
\end{itemize}
\end{lemm}
\begin{proof}
{\rm (1)}
 By Lemma \ref{lemm31} and {\rm (2)}, we have  $G_mV=0$ for all $m>r-1$.
 Now  suppose that $w=g_{\mathbf{i}}v_{\mathbf{i},c}$ with $\W(\mathbf{i})=q\in\Z$. For   $k >r$ and any fixed $F_k$,  by using Lie super-brackets  in \eqref{def55441.1}, we may transfer the only positive degree term in $[F_k, g_{\mathbf{i}}]$ to the right side, i.e.,
$[F_k,g_{\mathbf{i}}]\in \sum_{m\in\{q-k,\ldots,q\}}\U(\R_{-}\oplus\R_0)_{-m}\R_{k+m-q}$. So
\begin{eqnarray}\label{113.1222}
F_k w=[F_k,g_{\mathbf{i}}]v_{\mathbf{i},c}=g_{k-q}v_{\mathbf{i},c}+
\sum_{\mathbf{j}\in
\{\mathbf{k}\mid \mathbf{i}-\mathbf{k}\in \S,0\leq\W(\mathbf{k})\leq k-1\}}g_{\mathbf{i}-\mathbf{j}}v_{\mathbf{j},c}
\end{eqnarray}
for some $g_{k-q}\in \U(\R)_{k-q}$.

{\rm (2)} It follows from   {\rm (1)} that  we have   $\W(F_k g_{\mathbf{i}}v_{\mathbf{i},c})\leq\W(g_{\mathbf{i}}v_{\mathbf{i},c})-k+r$. The lemma clears.
\end{proof}

\begin{lemm}\label{lemm33}
 Let $m\in\Z,  c\in\C, \mathbf{i}\in\S$
  with $\widehat{k}=\mathrm{min}\{m\mid i_m\neq0\}\geq0$.   Assume  that $V$ is  a simple $\B$-module and there exists $r\in\Z_{\geq2}$ such that   $V$  satisfying  the conditions  \begin{itemize}
\item[{\rm (1)}] the action of $L_{r}$ on   $V$  is  injective;
\item[{\rm (2)}]  $L_{m}V=G_rV=0$ for all $m>r$.
\end{itemize}
Then
 \begin{itemize}
\item[{\rm (i)}] if $\widehat{k}=2k-1$ for some $k\in\Z_+$, we have
\begin{itemize}
\item[{\rm (a)}]  $\mathrm{deg}(L_{k+r}g_{\mathbf{i}}v_{\mathbf{i},c})=\mathbf{i}-\epsilon_{\widehat{k}}$;
\item[{\rm (b)}]
   $\mathbf{i}-\epsilon_{\widehat{k}}\notin\mathrm{supp}
   \big(L_{k+r}g_{\widetilde{\mathbf{i}}}v_{\widetilde{\mathbf{i}},c}\big)$ for all $\mathbf{i}\succ\widetilde{\mathbf{i}}$.
\end{itemize}
\item[{\rm (ii)}]
 if $\widehat{k}=2k$ for some $k\in\Z_+$, we have
\begin{itemize}
\item[{\rm (a)}]  $\mathrm{deg}(G_{k+r-1}g_{\mathbf{i}}v_{\mathbf{i},c})=\mathbf{i}-\epsilon_{\widehat{k}}$;
\item[{\rm (b)}]
   $\mathbf{i}-\epsilon_{\widehat{k}}\notin\mathrm{supp}
   \big(G_{k+r-1}g_{\widetilde{\mathbf{i}}}v_{\widetilde{\mathbf{i}},c}\big)$ for all $\mathbf{i}\succ\widetilde{\mathbf{i}}$.
   \end{itemize}
\end{itemize}
\end{lemm}
\begin{proof}
{\rm (i)} {\rm (a)}
To prove this, we   write  $L_{k+r}g_{\mathbf{i}}v_{\mathbf{i},c}$ as  \eqref{113.1222}. Clearly,     the only way to give
  $g_{\mathbf{i}-\epsilon_{\widehat{k}}}v_{\mathbf{i},c}$ is to commute $L_{r+k}$ with an  $L_{-k}$, which implies $\mathbf{i}-\epsilon_{\widehat{k}}\in\mathrm{supp}
   \big(L_{k+r}g_{\mathbf{i}}v_{\mathbf{i},c}\big)$.
If there exists  a   $G_{-k}$ in $g_{\mathbf{i}}$, we obtain
$[L_{r+k},G_{-k}]v_{\mathbf{i},c}=0.$
Then by Lemma \ref{lemm3231}, we conclude $\mathrm{deg}(L_{k+r}g_{\mathbf{i}}v_{\mathbf{i},c})=\mathbf{i}-\epsilon_{\widehat{k}}$.

{\rm (b)}
Now we  consider the following three cases.

 First consider  $\W(\widetilde{\mathbf{i}}) < \W(\mathbf{i})$. According to  Lemma \ref{lemm3231},  we have
$$\W(L_{k+r}g_{\widetilde{\mathbf{i}}}v_{\widetilde{\mathbf{i}},c})
\leq \W(\widetilde{\mathbf{i}})-k<\W(\mathbf{i}-\epsilon_{\widehat{k}})=\W(\mathbf{i})-k.$$
Obviously, {\rm (b)} follows in this case.

Assume $\W(\widetilde{\mathbf{i}})= \W(\mathbf{i})=p\in\Z$ and $\mathbf{D}(\widetilde{\mathbf{i}})< \mathbf{D}(\mathbf{i})$.
  If the element
$\mathbf{j}\in\mathrm{supp}
   \big(L_{k+r}g_{\widetilde{\mathbf{i}}}v_{\widetilde{\mathbf{i}},c}\big)$ is such that $\mathbf{D}(\mathbf{j})< \mathbf{D}(\widetilde{\mathbf{i}})$, then
$$\mathbf{D}(\mathbf{j})< \mathbf{D}(\widetilde{\mathbf{i}})\leq \mathbf{D}(\mathbf{i})-1=\mathbf{D}(\mathbf{i}-\epsilon_{\widehat{k}}).$$
This  shows  $\mathbf{j}\neq  \mathbf{i}-\epsilon_{\widehat{k}}$.
If  the element $\mathbf{j}\in\mathrm{supp}
   \big(L_{k+r}g_{\widetilde{\mathbf{i}}}v_{\widetilde{\mathbf{i}},c}\big)$  is such that $\mathbf{D}(\mathbf{j})=\mathbf{D}(\widetilde{\mathbf{i}})$, then
such $\mathbf{j}$ can only be given by commuting $L_{k+r}$ with some $L_{-j}$, where $j>k+r$.
Then we check
$$\W(\mathbf{j})=\W(\widetilde{\mathbf{i}})-k-r < \W\widetilde{(\mathbf{i}})-k= \W(\mathbf{i})-k=\W(\mathbf{i}-\epsilon_{\widehat{k}}),$$
 which   implies that  $\mathbf{j}\neq \mathbf{i}-\epsilon_{\widehat{k}}$. So,
{\rm (b)} also follows in this case.
 Let $\widetilde{k}=\mathrm{min}\{k\mid \widetilde{i}_k\neq0\}$ be in $\widetilde{\mathbf{i}}$.
  If $\widetilde{k}=\widehat{k}$, then by {\rm (1)}, $\mathrm{deg}(L_{k+r}g_{\widetilde{\mathbf{i}}}v_{\widetilde{\mathbf{i}},c})
  =\widetilde{\mathbf{i}}-\epsilon_{\widehat{k}}$, we also have {\rm (b)} in this case.

Consider the last case $\W(\widetilde{\mathbf{i}})=\W(\mathbf{i})= p$,  $\mathbf{D}(\widetilde{\mathbf{i}})=\mathbf{D}(\mathbf{i})$ and $\widetilde{k}>\widehat{k}$.
Then from Lemma \ref{lemm3231}, we have $\W(L_{k+r}g_{\widetilde{\mathbf{i}}}v_{\widetilde{\mathbf{i}},c})<p-k
=\W(\mathbf{i}-\epsilon_{\widehat{k}})$.
We complete  the proof of (i).

{\rm (ii)}
{\rm (a)} Write  $G_{k-1+r}g_{\mathbf{i}}v_{\mathbf{i},c}$ as the form of \eqref{113.1222}. We know  that  the only way to obtain
  $g_{\mathbf{i}-\epsilon_{\widehat{k}}}v_{\mathbf{i},c}$ is to commute $G_{k-1+r}$ with a  $G_{-k+1}$, which gives $\mathbf{i}-\epsilon_{\widehat{k}}\in\mathrm{supp}
   \big(G_{k-1+r}g_{\mathbf{i}}v_{\mathbf{i},c}\big)$.
Combining this with Lemma \ref{lemm3231}, we deduce   $\mathrm{deg}(G_{k-1+r}g_{\mathbf{i}}v_{\mathbf{i},c})=\mathbf{i}-\epsilon_{\widehat{k}}$.

{\rm (b)} Now we have the following three cases.

First consider   $\W(\widetilde{\mathbf{i}}) < \W(\mathbf{i})$. It follows from   Lemma \ref{lemm3231} that   we get
$$\W(G_{k-1+r}g_{\widetilde{\mathbf{i}}}v_{\widetilde{\mathbf{i}},c})
\leq \W(\widetilde{\mathbf{i}})-k+1<\W(\mathbf{i}-\epsilon_{\widehat{k}})
=\W(\mathbf{i})-k+1.$$
Thus, {\rm (b)} holds in this case.

Consider $\W(\widetilde{\mathbf{i}})= \W(\mathbf{i})=p\in\Z$ and $\mathbf{D}(\widetilde{\mathbf{i}})< \mathbf{D}(\mathbf{i})$.
  If  there exists
$\mathbf{j}\in\mathrm{supp}
   \big(G_{k-1+r}g_{\widetilde{\mathbf{i}}}v_{\widetilde{\mathbf{i}},c}\big)$  such that $\mathbf{D}(\mathbf{j})< \mathbf{D}(\widetilde{\mathbf{i}})$, then we have
$$\mathbf{D}(\mathbf{j})< \mathbf{D}(\widetilde{\mathbf{i}})\leq \mathbf{D}(\mathbf{i})-1=\mathbf{D}(\mathbf{i}-\epsilon_{\widehat{k}}).$$
This implies $\mathbf{j}\neq  \mathbf{i}-\epsilon_{\widehat{k}}$.
If there exists $\mathbf{j}\in\mathrm{supp}
   \big(G_{k-1+r}g_{\widetilde{\mathbf{i}}}v_{\widetilde{\mathbf{i}},c}\big)$    such that $\mathbf{D}(\mathbf{j})=\mathbf{D}(\widetilde{\mathbf{i}})$, then
such $\mathbf{j}$ can only be given by commuting $G_{k-1+r}$ with some $G_{-p}$, where $p>k-1+r$.
So we obtain
$$\W(\mathbf{j})=\W(\widetilde{\mathbf{i}})-k+1-r < \W(\widetilde{\mathbf{i}})-k+1= \W(\mathbf{i})-k+1=\W(\mathbf{i}-\epsilon_{\widehat{k}}),$$
 which   gives $\mathbf{j}\neq \mathbf{i}-\epsilon_{\widehat{k}}$. Obviously,
{\rm (ii)} also holds in this case.
We denote by $\widetilde{k}=\mathrm{min}\{k\mid \widetilde{i}_k\neq0\}$ for $\widetilde{\mathbf{i}}$.
  If $\widetilde{k}=\widehat{k}$, then by {\rm (1)}, $\mathrm{deg}(G_{k-1+r}g_{\widetilde{\mathbf{i}}}v_{\widetilde{\mathbf{i}},c})
  =\widetilde{\mathbf{i}}-\epsilon_{\widehat{k}}$, we also have {\rm (ii)} in this case.

At last, we   suppose   $\W(\widetilde{\mathbf{i}})=\W(\mathbf{i})= p\in\Z$,  $\mathbf{D}(\widetilde{\mathbf{i}})=\mathbf{D}(\mathbf{i})$ and $\widetilde{k}>\widehat{k}$.
Then it follows from   Lemma \ref{lemm3231} that  we get $\W(G_{k-1+r}g_{\widetilde{\mathbf{i}}}v_{\widetilde{\mathbf{i}},c})<p-k+1
=\W(\mathbf{i}-\epsilon_{\widehat{k}})$.
In conclusion, (b) holds.
\end{proof}
Now we present the   main result of this section.
\begin{theo}\label{th1}
 Let $m\in\Z, c\in\C$. Let
 $V$  be  a simple $\B$-module and  there exists $r\in\Z_{\geq2}$ such that $V$ satisfying  the following two conditions
  \begin{itemize}
\item[{\rm (1)}] the action of $L_{r}$ on   $V$  is  injective;
\item[{\rm (2)}]  $L_{m}V=G_rV=0$ for all $m>r$.
\end{itemize} Then we obtain that
   $\R$-module $\mathrm{Ind}_{\B,c}(V)$ is   simple.
\end{theo}
 \begin{proof}
 Let $0\neq w\in\mathrm{Ind}_{\B,c}(V)$ and $\mathrm{deg}(w)=\mathbf{i}$ for $\mathbf{i}\in\S$. Write  $\widehat{k}=\mathrm{min}\{m\mid i_m\neq0\}\geq0$. Based on Lemma \ref{lemm33}, we immediately
get the following results. If $\widehat{k}=2k-1$ for some $k\in\Z_+$, then   $L_{r+k}w\neq0$.
    If $\widehat{k}=2k$ for some $k\in\Z_+$, then $G_{r+k-1}w\neq0$. Therefore, from any  $0\neq w\in\mathrm{Ind}_{\B,c}(V)$ we always get a nonzero element in
 $\U(\R )w\cap V\neq0$, which shows the simplicity of $\mathrm{Ind}_{\B,c}(V)$. This completes  the proof.
\end{proof}

\section{Characterization of simple restricted $\R$-modules}
  For $m\in\Z, r\in\Z_{\geq2}$,
we denote   $$\R^{(r)}=\bigoplus_{m> r}\C L_{m}\oplus\bigoplus_{m> r-1}\C G_{m}.$$
First,  several equivalent conditions of simple
restricted modules over $\R$ are shown.
\begin{prop}\label{th22}
 Let    $\P$ be a simple   module for $\R$.  Then the following conditions are equivalent:
\begin{itemize}
\item[{\rm (i)}] $\P$ is a restricted $\R$-module.

\item[{\rm (ii)}]   There exists $r\in\Z_{\geq2}$ such that the actions of $L_{m},G_n$ for $m>r, n>r-1$ on $\P$ are locally nilpotent.

\item[{\rm (iii)}]
There exists $r\in\Z_{\geq2}$ such that the actions of $L_{m},G_n$ for   $m>r, n>r-1$ on $\P$ are locally finite.

\item[{\rm (iv)}] There exists $r\in\Z_{\geq2}$ such that  $\P$ is a  locally nilpotent $\R^{(r)}$-module.

\item[{\rm (v)}]       There exists $r\in\Z_{\geq2}$ such that  $\P$ is a  locally finite $\R^{(r)}$-module.
\end{itemize}
\end{prop}

\begin{proof}

It is clear that $(\mathrm{iv})\Rightarrow(\mathrm{v})\Rightarrow(\mathrm{iii})$ and  $(\mathrm{iv})\Rightarrow(\mathrm{ii})\Rightarrow(\mathrm{iii})$.
So, we only need to prove $(\mathrm{i})\Rightarrow(\mathrm{iv})$ and
$(\mathrm{iii})\Rightarrow(\mathrm{i})$.

$(\mathrm{i})\Rightarrow(\mathrm{iv})$.
From the definition of restricted module, for any nonzero element $v\in\P$, there exists
$r^\prime\in\Z_+$ such that $L_{m}v=G_nv=0$ for $m>r^\prime,n>r^\prime-1$. By the simplicity of $\P$, we have  $\P=\U(\R)v$. Then from the $\mathrm{PBW}$ Theorem,  we check that $\P$ is a locally
nilpotent  module over $\R^{(r)}$ for $r>r^\prime$.

$(\mathrm{iii})\Rightarrow(\mathrm{i})$.
Since there exists $r\in\Z_{\geq2}$ such that the actions of
$L_{m}$ and $G_{n}$   for all $m>r, n>r-1$ on $\P$ are locally finite,
then we can choose a nonzero element $w\in \P$ such that $L_{r+1}w=\mu w$ for some $\mu\in \C$.

Let  $m,n\in\Z,r\in\Z_{\geq2}$ with $m>r,n>r-1$. We  denote
\begin{eqnarray*}
 &&\M_L=\sum_{i\in\N}\C L_{r+1}^iL_{m}w=\U(\C L_{r+1})L_{m}w
 \quad \mathrm{and}\quad \M_G=\sum_{i\in\N}\C L_{r+1}^iG_{n}w=\U(\C L_{r+1})G_{n}w.
\end{eqnarray*}
 It follows from the definition of $\R$ and  $i\in\N$ that  we get
\begin{eqnarray*}
&&L_{m+i(r+1)}w\in \M_L\Rightarrow L_{m+(i+1)(r+1)}w\in \M_L,\quad \forall \ m>r,
\\&&G_{n+i(r+1)}w\in \M_G\Rightarrow G_{n+(i+1)(r+1)}w\in \M_G,\quad \forall \ n>r-1.
\end{eqnarray*}
Then we have
  $L_{m+i(r+1)}w\in \M_L$ and $G_{n+i(r+1)}w\in \M_G$ for all $m>r, n>r-1$ by  induction on $i\in\N$.
In particular,  $\sum_{i\in\N}\C L_{m+i(r+1)}w$ and $\sum_{i\in\N}\C G_{n+i(r+1)}w$ are both
finite-dimensional for    $m>r, n>r-1$. Thus,
 \begin{eqnarray*}
&&\sum_{i\in\N}\C L_{r+1+i}w=\C L_{r+1}w+\sum_{m=r+2}^{2r+2}\big(\sum_{i\in\N}\C L_{m+i(r+1)}w\big),
\\&&\sum_{i\in\N}\C G_{r+i}w=\C G_{r}w+\sum_{n=r+1}^{2r+1}\big(\sum_{i\in\N}\C G_{n+i(r+1)}w\big)
\end{eqnarray*}
 are both finite dimensional. Then  we can safely choose $t\in\N$ such that
$$
\sum_{i\in\N}\C L_{r+1+i}w=\sum_{i=0}^{t}\C L_{r+1+i}w
\quad \mathrm{and} \quad \sum_{i\in\N}\C G_{r+i}w=\sum_{i=0}^{t}\C G_{r+i}w.
$$
We denote  $$M^\prime=\sum_{x_1,\ldots,x_t\in\N,y_0,y_1,\ldots,y_t\in\X}\C L_{r+1}^{x_1}\cdots
 L_{r+t}^{x_t} G_{r}^{y_0}G_{r+1}^{y_1}\cdots
 G_{r+t}^{y_t}  w,$$
which is a (finite-dimensional) $\R^{(r)}$-module by $(\mathrm{i})$.

It follows that we can  take
 a minimal $q\in\N$ such that
 \begin{equation}\label{lm3.3}
 (L_m+\alpha_1L_{m+1}+\cdots + \alpha_{q}L_{m+q})M^\prime=0
 \end{equation}
 for some $m\in\Z$ with $m> r$ and  $\alpha_i\in \C,i=1,\ldots,q$.
By applying $L_{m}$ to \eqref{lm3.3}, we immediately get
$$(\alpha_1[L_{m},L_{m+1}]+\cdots +\alpha_{q}[L_{m},L_{m+q}])M^\prime=0.$$
This implies $q=0$, i.e., $L_mM^\prime=0$  for some $m> r$.
Thus, for any $k>m$, one checks
 \begin{eqnarray*}
 &&L_{m+k}M^\prime=\frac{1}{m-k}\big(L_kL_m-L_mL_k\big)M^\prime=0.
 \end{eqnarray*}
Choosing any $l>\mathrm{max}\{\frac{m}{2},r-1\}$, we  obtain
\begin{eqnarray*}
 &&G_{m+l}M^\prime=\frac{2}{2l-m}\big( L_mG_{l}-G_{l}L_{m}\big)M^\prime=0.
 \end{eqnarray*}
Therefore,  there exists a nonzero element  $u\in\P$ such that $L_ku=G_ku=0$  for all
$k>2m$. Note that $\P=\U(\R)u$. By the $\mathrm{PBW}$ Theorem,  we conclude that
each element of $\P$ can be written as a linear combinations of vectors
$$\cdots G_{2m-1}^{l_{2m-1}}G_{2m}^{l_{2m}}\cdots L_{2m-1}^{j_{2m-1}}L_{2m}^{j_{2m}}u.$$
Then for any $u^\prime \in \P$, there exists sufficiently large $s\in\Z_+$ such that $L_nu^\prime=G_nu^\prime=0$ for any $n>s$, which shows that $\P$ is a restricted $\R$-module. Then (i) holds. We complete the proof.
\end{proof}

\begin{theo}\label{th2}
 Let $c\in\C$ and  $\P$ be a simple restricted  module for $\R$. Assume that there exists $k\in\Z_{\geq2}$ such that the action of $L_{k}$ on   $\P$  is  injective.
 \begin{itemize}
\item[{\rm (1)}] Then  there
exists the smallest $b\in\Z_{\geq2}$ with $b\geq k$ such that
 $$
 \M_{b}=\Big\{w\in \P\mid L_{m}w=G_{n}w=0 \quad \mbox{for\ all}\ m>b,n> b-1\Big\}\neq0.$$ In particular, $V:=\M_{b}$ is a $\B$-module.
 \item[{\rm (2)}] More important,   $\P$ can be  described by $V$ as follows. \begin{itemize}
\item[{\rm (i)}] If $b=k$, then $V$ is a simple $ \mathcal{B}$-module, and
$\P\cong\mathrm{Ind}_{\B,c}(V)$.

\item[{\rm (ii)}]   If $b>k$, suppose that $L_b$ acts injectively on $V$.  Then $V$ is a simple $ \mathcal{B}$-module, and
$\P\cong\mathrm{Ind}_{\B,c}(V)$.
\end{itemize}
   \end{itemize}
\end{theo}
\begin{proof}
(1) Note that there exists $k\in\Z_{\geq2}$ such that the action of $L_{k}$ on   $\P$  is  injective. Then from the proof of   Proposition \ref{th22},  we see that
  $$\M_{b^\prime}=\Big\{w\in \P\mid L_{m}w=G_{n}w=0 \quad \mbox{for\ all}\ m>b^\prime,n> b^\prime-1\Big\}\neq0$$ for sufficiently large $b^\prime\geq k$.
On the other hand, $\M_{{b^\prime}}=0$ for all ${b^\prime}<k$ since the action of  $L_k$ on $\P$ is injective. Thus we  can find the smallest   $b\in\Z_{\geq2}$ with $b\geq k$  such that $V:=\M_b\neq0$.

It follows from  $m>b,n>b-1,  p\in\N, q\in\Z_+$
that we  have
\begin{eqnarray*}
&&L_{m}(L_{p}w)=(p-m)L_{m+p}w=0,
\
G_{n}(L_{p}w)=
(\frac{p}{2}-n)G_{n+p}w=0,
\\&&L_{m}(G_{q}w)=(q-\frac{m}{2})G_{m+q}w=0,
\
G_{n}(G_{q}w)=
-2L_{n+q}w=0.
\end{eqnarray*}
 This shows that $L_pw,G_{q}w\in V$ for all $p\in\N, q\in\Z_+$.
   So, $V$ is a  $\mathcal{B}$-module.

(2) First, consider (i).     Clearly,   the action of $L_b$    on $V$ is injective.
\begin{clai}
For $s_1,s_2\in\{1,\ldots,b-1\}$ with $s_1+s_2=b$, we have $G_{s_1}V\neq0$ and $G_{s_2}V\neq0$.
\end{clai}
Suppose $G_{s_1}V=0$ or $G_{s_2}V=0$. For any $b\in\Z_{\geq2}$, we have  $$L_bV=-\frac{1}{2}(G_{s_1}G_{s_2}+G_{s_2}G_{s_1})V=0,$$ which shows  $V=0$. This contradicts with  $V\neq0$ in (1). The claim holds.
Then one can  see that  the elements in $V$ are not the same as
in $$\big\{w\in \P\mid L_{m}w=G_{n}w=0 \quad \mbox{for\ all}\ m>b,n> a^\prime\big\}\neq0\quad \mathrm{for}\quad a^\prime<b-1.$$

Since $\P$ is simple
and generated by $V$,  we know that  there exists a canonical surjective map
$$\Phi:\mathrm{ Ind}_{\B,c}(V) \rightarrow \P, \quad \Phi(1\otimes v)=v,\quad \forall  v\in V.$$
Hence, we only need to show that $\Phi$ is  an injective map.  Let $\mathcal{K}=\mathrm{ker}(\Phi)$. For any $v\in V$ we have $\Phi(1\otimes v)=v$, then   $\mathcal{K}\cap (1\otimes V)=0$. If	$\mathcal{K}\neq0$, we can choose $0\neq w\in \mathcal{K}\setminus (1\otimes V)$ such that $\mathrm{deg}(w)=\mathbf{i}$ is minimal possible.
Observe    that $\mathcal{K}$ is an $\R$-submodule of $\mathrm{Ind}_{\B,c}(V)$.
 Consider that $L_b$ acts injectively on $V$. Then by using  the similar method  in Lemma \ref{lemm33},  a new vector $\eta\in \mathcal{K}$  with $\mathrm{deg}(\eta)\prec\mathbf{i}$ can be obtained.  This  shows a contradiction, namely,  $\mathcal{K}=0$. Then
 $\P\cong \mathrm{Ind}_{\B,c}(V)$. By the property of induced modules, we see that
$V$ is simple as a  $\B$-module.

Using  an identical
process of  (i), we have (ii).
This completes  the proof.
\end{proof}

 \section{Weak modules   for   vertex operator
superalgebras}
 \textit{The    Neveu-Schwarz algebra} $\mathcal{N}$   is the Lie superalgebra with a basis
$\{L_m,G_{p}, C\mid m\in\Z,p\in\frac{1}{2}+\Z\}$
and the Lie super-bracket  defined by
 \begin{equation*}\label{def1.1}
\aligned
&[L_m,L_n]= (n-m)L_{m+n}+\frac{n^{3}-n}{12}\delta_{m+n,0}C,\\&
 [G_p,G_q]= -2L_{p+q}+\frac{4p^{2}-1}{12}\delta_{p+q,0}C,\\&
 [L_m,G_p]= (p-\frac{m}{2})G_{m+p},\ [\mathcal{N},C]=0
\endaligned
\end{equation*}
for $m,n\in\Z$, $p,q\in\frac{1}{2}+\Z$.
By its definition, we have the following decomposition:
$$\mathcal{N}=\mathcal{N}_{\bar 0}\oplus\mathcal{N}_{\bar 1},$$
where $\mathcal{N}_{\bar 0}=\mathrm{span}\{L_m,C\mid m\in\Z\}$, $\mathcal{N}_{\bar 1}=\mathrm{span}\{G_p\mid p\in\frac{1}{2}+\Z\}$.
Note that the $\mathcal{N}$  is isomorphic to the subalgebra of $\R$ spanned by
$\{L_{m}\mid m\in2\Z\}\cup\{G_p\mid p\in2\Z+1\}\cup\{C\}$.
It is clear that the $\mathcal{N}$  has a $\frac{1}{2}\Z$-grading by the eigenvalues of the adjoint action
of $L_0$.  Then $\mathcal{N}$ has the following triangular decomposition:
$$\mathcal{N}=\mathcal{N}_+\oplus\mathcal{N}_0\oplus\mathcal{N}_-,$$ where $\mathcal{N}_+=\mathrm{span}\{L_m,G_p\mid m,p>0\}$, $\mathcal{N}_-=\mathrm{span}\{L_m,G_p\mid m,p<0\}$ and $\mathcal{N}_0=\C\{L_0,C\}$.

Set
\begin{eqnarray}\label{LTG511}
L(z)=\sum_{m\in\Z}L_mz^{-m-2},\ G(z)=\sum_{n\in\Z}G_{n+\frac{1}{2}}z^{-n-2}.
\end{eqnarray}
By Section $4.2$ of \cite{L}, we have
\begin{eqnarray*}
&&[L(z_1),L(z_2)]=z_1^{-1}\delta(\frac{z_2}{z_1})\frac{\partial}{\partial z_2}(L(z_2))
+2z_1^{-2}\frac{\partial}{\partial z_2}(\delta(\frac{z_2}{z_1}))L(z_2)+\frac{c}{12}z_1^{-4}(\frac{\partial}{\partial z_2})^3\delta(\frac{z_2}{z_1}),
\\&&[L(z_1),G(z_2)]=z_1^{-1}\delta(\frac{z_2}{z_1})\frac{\partial}{\partial z_2}(G(z_2))
+\frac{3}{2}\big(\frac{\partial}{\partial z_2}z_1^{-1}\delta(\frac{z_2}{z_1})\big)G(z_2),
\\&&[G(z_1),G(z_2)]=2z_1^{-1}\delta(\frac{z_2}{z_1})L(z_2)
+\frac{c}{3}(\frac{\partial}{\partial z_2})^2z_1^{-1}\delta(\frac{z_2}{z_1}).
\end{eqnarray*}
\begin{defi}\rm
A {\em vertex  superalgebra}  denoted
   by   a quadruple $(V,Y,\mathbf{1},D)$  is a $\Z_2$-graded
vector space
$$V=V^{(0)}\oplus V^{(1)},$$
and equipped with a linear map
 \begin{eqnarray*}
   V&\longrightarrow& (\mathrm{End}(V))[[z,z^{-1}]]
\\ v&\longmapsto& Y(v,z)=\sum_{n\in\Z}v_nz^{-n-1}\quad(\mathrm{where}\ v_n\in\ \mathrm{End}(V)),
\end{eqnarray*}
and with a specified vector  $\mathbf{1}\in V_0$ (the {\it vacuum vector}) and an endomorphism $D$ of $V$,   such that
\begin{itemize}
\item[{\rm (1)}]
For any $u,v\in V$, $u_nv = 0 \ \mathrm{for}\ n \ \mathrm{sufficiently\ large}$;
\item[{\rm (2)}]
$[D,Y(v,z)]=Y(D(v),z)=\frac{\mathrm{d}}{\mathrm{dz}}Y(v,z)$ for any $v\in V$;
\item[{\rm (3)}]
$Y(\mathbf{1},z)=\mathrm{Id}_V$ (the {\it identity operator} of $V$);
\item[{\rm (4)}]
$Y(v,z)\mathbf{1}\in \mathrm{End}(V)[[z]] \ \mathrm{and}\ \mathrm{lim}_{z\rightarrow0} Y(v,z)\mathbf{1}=v$ for any $v\in V$;
\item[{\rm (5)}]
$z_0^{-1}\delta(\frac{z_1-z_2}{x_0})Y(u,z_1)Y(v,z_2)
-(-1)^{|u||v|}z_0^{-1}\delta(\frac{z_2-z_1}{-z_0})Y(v,z_2)Y(u,z_1))
\\=z_2^{-1}\delta(\frac{z_1-z_0}{z_2})Y(Y(u,z_0)v,z_2)$ (the {\it Jacobi  identity}),
  where $|v|=j$ if $v\in V^{(j)}$
for $j\in\Z_2$.
\end{itemize}
This completes the definition of   vertex  superalgebra.
\end{defi}
A vertex superalgebra $V$ is called a {\em vertex operator superalgebra} if there
exists  another distinguished vector $\omega$ of V satisfying the following conditions
\begin{itemize}
\item[{\rm (6)}]
$[L(m),L(n)]= (n-m)L(m+n)+\frac{n^{3}-n}{12}\delta_{m+n,0}C$ for $m,n\in\Z$, where $Y(\omega,z)=\sum_{n\in\Z}L(n)z^{-n-2}$;
\item[{\rm (7)}]
$L_{-1}=D$, i.e., $\frac{\mathrm{d}}{\mathrm{dz}}Y(v,z)=Y(L_{-1}v,z)$ for any $v\in V$;
\item[{\rm (8)}]
$V$ is $\frac{1}{2}\Z$-graded such that $V=\bigoplus_{n\in\frac{1}{2}\Z}V_{(n)}$, $L(0)\mid_{V_{(n)}}=n\mathrm{Id}_{V_{(n)}}$, $\mathrm{dim}(V_{(n)})<\infty$ and $V_{(n)}=0$ for
$n$  sufficiently negative.
\end{itemize}

For any $h,c\in\C$, assume that $W(h,c)$ is the Verma module for $\mathcal{N}$ with  highest weight $(h,c)$. Let $\mathbf{1}$ be a highest weight
vector of $W(0,c)$.
We denote
\begin{eqnarray}\label{5342}
\bar{W}(0,c)=W(0,c)/\langle G_{-\frac{1}{2}}\mathbf{1}\rangle,
\end{eqnarray}
 where  $\langle G_{-\frac{1}{2}}\mathbf{1}\rangle$
 is the submodule generated by $G_{-\frac{1}{2}}\mathbf{1}$. It is well know that $\bar{W}(0,c)$ has a natural vertex operator superalgebra
structure (see \cite{KW,L}).

Assume that $\V$ is  a vertex superalgebra. Define  the following linear map
 \begin{eqnarray*}
\psi:\quad   \V&\longrightarrow& \V
\\ a+b&\longmapsto&a-b
\end{eqnarray*}
for $a\in \V^{(0)}, b\in \V^{(1)}$. It is clear that $\psi$  is an automorphism of $\V$ (called the {\em canonical
automorphism} (see \cite{FFF})). Then $\mathrm{Aut}(\bar W(0,c))=\Z_2=\langle \psi\rangle$.

The following results can be found in \cite{L,L1}.
\begin{lemm}\label{lemm51}
 Let $c\in\C$, $\bar{W}(0,c)$ defined as \eqref{5342}.
\begin{itemize}
\item[{\rm (i)}]
  Any weak  $\bar{W}(0,c)$-module is a  naturally restricted  module for $\mathcal{N}$  with central charge $c$; and conversely, any restricted  module for $\mathcal{N}$  with central charge $c$
is a weak  $\bar{W}(0,c)$-module;

\item[{\rm (ii)}]
Any  weak $\psi$-twisted $\bar W(0,c)$-module   is a  naturally restricted  module for $\R$ with central charge $c$; and conversely, any restricted module for
 $\R$ with central charge $c$ is a weak $\psi$-twisted $\bar W(0,c)$-module.
\end{itemize}
\end{lemm}

The following  result appeared in \cite{LPX55}.
\begin{theo}\label{prop3.5}
Let       $H$ be a simple module of $\mathfrak{B}=\mathcal{N}_+\oplus\mathcal{N}_0$. Assume that the action  of $C$ on   $H$ is a scalar  $c$.   If  there exists $t\in\N$ such that $H$ satisfying  the following two conditions
  \begin{itemize}
\item[{\rm (1)}] the action of $L_{t}$ on   $H$  is  injective;
\item[{\rm (2)}]  $L_{m}H=0$ for all $m>t$,
\end{itemize} then   $\mathrm{Ind}_{\mathfrak{B},c}(H)$ is a simple $\mathcal{N}$-module.
\end{theo}
For any   $t\in\N$, let $\mathfrak{N}_t$ be the set of simple $\mathfrak{B}$-modules
satisfying the condition in   Theorem \ref{prop3.5}.   For any   $r\in\Z_{\geq2}$, let $\mathfrak{R}_r$ be the set of simple $\B$-modules
satisfying the condition in   Theorem \ref{th1}. Set
$\mathfrak{N}=\bigcup_{t\in\N}\mathfrak{N}_t$ and $\mathfrak{R}=\bigcup_{r\in\Z_{\geq2}}\mathfrak{R}_r$.
By \cite{LPX55},  Theorem \ref{th2} and Lemma \ref{lemm51}, we have the   following results.
\begin{prop}
\begin{itemize}
\item[{\rm (1)}] The set of $\{\mathrm{Ind}_{\mathfrak{B},c}(H)\mid c\in\C, H\in \mathfrak{N}\}$ gives a complete list  of simple  weak  $\bar{W}(0,c)$-module under some conditions.
\item[{\rm (2)}] The set of $\{\mathrm{Ind}_{\B,c}(V)\mid c\in\C, V\in \mathfrak{R}\}$ gives a complete list of  simple weak $\psi$-twisted $\bar W(0,c)$-modules under some conditions.
 \end{itemize}
\end{prop}

\section{Examples}
In this section, we show some examples of simple restricted  $\R$-modules.

\subsection{Simple induced  modules}
Let $\mathfrak{l}= \C x+\C y$ be  the $2$-dimensional
solvable Lie algebra with basis $\{x,y\}$, which   satisfies the non-trivial Lie bracket $[x,y]=y$.
 Basically,  $\mathfrak{l}$ is a  subalgebra of classical $3$-dimensional
  Lie algebra $\mathfrak{sl}(2)$.  We  construct  a class of  induced restricted $\R$-module  by using  a $\C[y]$-torsion-free simple $\mathfrak{l}$-module $\mathfrak{k}=(\partial-1)^{-1}\C[\partial^{\pm1}]$ defined  in \cite[Example 13]{LMZ}, whose
structure is given by $$x\cdot f(\partial)=\partial\frac{d}{d\partial}f(\partial)+\frac{f(\partial)}{\partial^2(\partial-1)},\ y\cdot f(\partial)=\partial f(\partial),\ \forall
f(\partial)\in\mathfrak{k}.$$
Denote
$\widehat{\B}=\bigoplus_{m\geq0}\C L_m\oplus\bigoplus_{n\geq2}\C G_n.$ Then we can extend  $\mathfrak{l}$-module  to a $\widehat{\B}$-module   $\widehat{V}_{\mathfrak{k}}=(\partial-1)^{-1}\C[\partial^{\pm1}]
\oplus G_1(\partial-1)^{-1}\C[\partial^{\pm1}]$    by defining
\begin{eqnarray*}
&&L_0\cdot f(\partial)=2x\cdot f(\partial),\ L_0\cdot (G_1f(\partial))=G_1\big(2x\cdot f(\partial)+f(\partial)\big),
\\&& L_{m}\cdot f(\partial)=G_n\cdot f(\partial)= L_{m}\cdot (G_1f(\partial))=G_n\cdot (G_1f(\partial))=0,
\\&& L_2\cdot f(\partial)=y\cdot f(\partial),\ L_2\cdot (G_1f(\partial))=G_1\big(y\cdot f(\partial)\big),
\ C\cdot f(\partial)=cf(\partial),
\end{eqnarray*}
where  $c\in\C,m\in\Z_{\geq3}\bigcup\{1\},n\in \Z_{\geq2},f(\partial)\in\mathfrak{k}$.
Note that $G_1^2f(\partial)=-y\cdot f(\partial)$. Clearly,  $V_{\mathfrak{k}}=U(\B)\otimes_{U(\widehat{\B})}\widehat{V}_\mathfrak{k}$ is a  simple $\B$-module.
By Theorem \ref{th1}, we get the simple induced $\R$-modules $\mathrm{Ind}_{\B,c}(V_{\mathfrak{k}})$.

\subsection{Whittaker modules}
For $m,n\in\Z$, we denote
$$\widehat{\b}=\bigoplus_{m\geq1}\C L_m\oplus\bigoplus_{n\geq2}\C G_n.$$
Let  $\phi:\widehat{\b}\rightarrow\C$ be a non-trivial Lie superalgebra homomorphism and $\phi(G_2)=0$.  Then we have  $\phi(L_m)=\phi(G_n)=0$ for $m>2$, $n>1$. Let $\mathfrak{s}_\phi=\C v_{\bar0}\oplus\C v_{\bar1}$ be  a $2$-dimensional vector space with
$$xv_{\bar0}=\phi(x)v_{\bar0}, \ v_{\bar1}=G_1v_{\bar0}, \  Cv_{\bar0}=cv_{\bar0}, \ Cv_{\bar1}=cv_{\bar1}$$  for all $x\in \widehat{\b}.$
Clearly, if $\phi(L_{2})\neq0$,   $\mathfrak{s}_{\phi}$ is a simple $\widehat{\b}$-module and  $\mathrm{dim}(\mathfrak{s}_{\phi})=2$.
Now we consider the induced module
$$M_\phi=\U(\widehat{\B})\otimes_{\U(\widehat{\b})} \mathfrak{s}_\phi=\C[L_0]v_{\bar0}\oplus \C[L_0](G_1v_{\bar0}).$$
 Let $V_{\phi}=\U(\B)\otimes_{\U(\widehat{\B})}M_\phi$. It is easy to check that $V_{\phi}$ is a simple $\B$-module if $\phi(L_{2})\neq0$.
 When $\phi(L_{2})\neq0$,
  the simple induced $\R$-modules $\mathrm{Ind}_{\B,c}(V_{\phi})$  in Theorem \ref{th1}  are so-called classical Whittaker modules (see \cite{LPX}).

\subsection{High order Whittaker  modules}
  In this section,  we show   a generalization version of Whittaker modules   of $\R$ called the high order Whittaker modules.

For $m,n\in\Z$, $s\in\Z_{\geq2}$,
we denote   $$\Gamma{(s)}=\bigoplus_{m\geq s}\C L_{m}\oplus\bigoplus_{n\geq s}\C G_{n}.$$
 Let $\phi_s$ be a Lie superalgebra homomorphism $\phi_s:\Gamma{(s)}\rightarrow\C$ for  $s\in\Z_{\geq2}$.
Then we get   $\phi_s(L_m)=\phi_s(G_n)=0$ for $m>2s, n>2s-1$.
Assume that $\mathfrak{s}_{\phi_s}=\C v_{\bar0}\oplus\C v_{\bar1}$ is a $2$-dimensional vector space with
$$xv_{\bar0}=\phi(x)v_{\bar0},\ v_{\bar1}=G_1v_{\bar0}, \ Cv_{\bar0}=cv_{\bar0},\ Cv_{\bar1}=cv_{\bar1}$$  for all $x\in \Gamma{(s)}.$
If $\phi_s(L_{2s})\neq0$,  $\mathfrak{s}_{\phi_s}$ is a simple $\Gamma{(s)}$-module and  $\mathrm{dim}(\mathfrak{s}_{\phi_s})=2$.
 Consider the induced module
$$M_{\phi_s}=\U(\widehat{\mathcal{B}})\otimes_{\U(\Gamma{(s)})} \mathfrak{s}_{\phi_s}.$$
  Denote $V_{\phi_s}=\U(\B)\otimes_{\U(\widehat{\B})}M_{\phi_s}$. It is clear that  $V_{\phi_s}$ is a simple $\B$-module if $\phi(L_{2s})\neq0$.
The corresponding simple  $\R$-modules $\mathrm{Ind}_{\B,c}(V_{\phi_s})$  in Theorem \ref{th1}  are exactly the high order Whittaker modules.


\section*{Acknowledgements}
This work was carried out during the  author's visit to University of California, Santa Cruz. The author would like to
thank Professor Chongying Dong  for the warm hospitality during his visit.
This work was supported by the National Natural Science Foundation of China
(Grant Nos. 11801369,  12171129), the Overseas Visiting Scholars Program  of Shanghai Lixin University of  Accounting and Finance (Grant No. 2021161). The
 author thanked the Professor Hongyan Guo for telling him   reference   \cite{L1}. 

\bigskip

Haibo Chen
\vspace{2pt}

  1. School of  Statistics and Mathematics, Shanghai Lixin University of  Accounting and Finance,   Shanghai
201209, China

\vspace{2pt}
2. Department of Mathematics, Jimei University, Xiamen, Fujian 361021, China

\vspace{2pt}
Hypo1025@163.com

\end{document}